\documentclass[a4paper,9pt,leqno]{amsart}
\usepackage[cp1250]{inputenc}
\usepackage{ifthen}
\usepackage{amssymb}
\usepackage{amsmath}
\usepackage{amsfonts}
\usepackage{latexsym}

\numberwithin{equation}{section}
\newtheorem{theo}{Theorem}[section]

\newtheorem{cor}{Corollary}[section]

\input{tcilatex}

\begin{document}
\title[Some sufficient conditions for the univalence of an integral operator]%
{Some sufficient conditions for the univalence of an integral operator}
\author{$^{\dagger }$Halit Orhan}
\address{$^{\dagger }$Department of Mathematics, Faculty of Science, Ataturk
University, 25240, Erzurum, Turkey.}
\email{horhan@atauni.edu.tr (Corresponding author)}
\author{$^{\dagger\dagger }$Dorina Răducanu}
\address{$^{\dagger\dagger }$Faculty of Mathematics and Computer Science,
Transilvania University of Braşov, 50091, Iuliu Maniu, 50, Braşov, Romania.}
\email{dorinaraducanu@yahoo.com}
\author{$^{\dagger\dagger\dagger }$Murat Ça\u{g}lar}
\address{$^{\dagger\dagger\dagger }$Department of Mathematics, Faculty of
Science, Ataturk University, 25240, Erzurum, Turkey.}
\email{mcaglar25@gmail.com}

\begin{abstract}
Making use of the method of subordination chains, we obtain some sufficient
conditions for the univalence of an integral operator. In particular, as
special cases, our results imply certain known univalence criteria. A
refinement to a quasiconformal extension criterion of the main result, is
also obtained.
\end{abstract}

\maketitle

\section{Introduction}

Denote by $\mathcal{U}_{r}\;(0<r\leq1)$ the disk of radius $r$ centered at $%
0 $, i.e $\mathcal{U}_{r}=\left\{z\in\mathbb{C}:|z|<r\right\}$ and let $%
\mathcal{U}=\mathcal{U}_{1}$ be the unit disk.

Let $\mathcal{A}$ denote the class of analytic functions in $\mathcal{U}$
which satisfy the usual normalization 
\begin{equation*}
f(0)=f^{\prime }(0)-1=0.
\end{equation*}

One of the most important univalence criterion for functions in the class $%
\mathcal{A}$ was obtained by Becker in 1972 \cite{Be1}. His result was
derived by means of Loewner chains and Loewner differential equation. During
the time many extensions of Becker's criterion have been given, among them
being the results due to Ahlfors \cite{Ah}, Lewandowski \cite{Le}, Pascu 
\cite{Pa1}, \cite{Pa2}, Ruscheweyh \cite{Ru}, Ovesea \cite{Ov} and Kanas and
Srivastava \cite{KaSri}.

In the present paper we use the method of subordination chains to obtain
some sufficient conditions for the univalence of an integral operator. Our
results generalize certain criteria obtained by Pascu \cite{Pa2}, Danikas
and Ruscheweyh \cite{DaRu}, Moldoveanu \cite{Mo}, Deniz and Orhan \cite{DeOr}%
, Răducanu et. all \cite{RaOrDe}. Also, we obtain a refinement to a
quasiconformal extension criterion of the main result.

\section{Loewner chains and quasiconformal extensions}

Before proving our main theorem we need a brief summary of Loewner chains
and Becker's method of constructing quasiconformal extensions by means of
Loewner chains and generalized Loewner differential equation.

A function $L(z,t):\mathcal{U}\times[0,\infty)\rightarrow\mathbb{C}$ is said
to be a \textit{subordination chain} or a \textit{Loewner chain} if:

\begin{itemize}
\item[(i)] $L(z,t)$ is analytic and univalent in $\mathcal{U}$ for all $%
t\geq0$.

\item[(ii)] $L(z,t)\prec L(z,s)$ for all $0\leq t\leq s<\infty$, where the
symbol $"\prec"$ stands for subordination.
\end{itemize}

The following result, due to Pommerenke, is often used to prove univalence
criteria.

\begin{theo}
\textrm{(\cite{Po1}, \cite{Po2})} Let $L(z,t)=a_{1}(t)z+\ldots$ be an
analytic function in $\mathcal{U}_{r}\;(0<r\leq1)$ for all $t\geq0$. Suppose
that:

\begin{itemize}
\item[(i)] $L(z,t)$ is a locally absolutely continuous function of $%
t\in[0,\infty)$, locally uniform with respect to $z\in\mathcal{U}_{r}$.

\item[(ii)] $a_{1}(t)$ is a complex valued continuous function on $[0,\infty
)$ such that $a_{1}(t)\neq 0$, $\lim_{t\rightarrow \infty }|a_{1}(t)|=\infty 
$ and 
\begin{equation*}
\left\{ \frac{L(z,t)}{a_{1}(t)}\right\} _{t\geq 0}
\end{equation*}%
is a normal family of functions in $\mathcal{U}_{r}$.

\item[(iii)] There exists an analytic function $p:\mathbb{U}%
\times[0,\infty)\rightarrow\mathbb{C}$ satisfying $\Re p(z,t)>0$ for all $%
(z,t)\in\mathcal{U}\times[0,\infty)$ and

\begin{equation}  \label{2.1}
z\frac{\partial L(z,t)}{\partial z}=p(z,t)\frac{\partial L(z,t)}{\partial t}%
\;,z\in\mathcal{U}_{r}\;,\text{a.e}\;\;t\geq0.
\end{equation}
\end{itemize}

Then, for each $t\geq0$, the function $L(z,t)$ has an analytic and univalent
extension to the whole disk $\mathbb{U}$, i.e $L(z,t)$ is a subordination
chain.
\end{theo}

Let $k$ be constant in $[0,1)$. Recall that a homeomorphism $f$ of $G\subset 
\mathbb{C}$ is said to be $k-$\textit{quasiconformal} if $\partial _{z}f$
and $\partial _{\overline{z}}f$ ,in the distributional sense, are locally
integrable on $G$ and fulfill $\left\vert \partial _{\overline{z}%
}f\right\vert \leq k\left\vert \partial _{z}f\right\vert $ almost everywhere
in $G.$

An important problem in the theory of univalent functions is to find
functions that have quasiconformal extensions to $\mathbb{C}$.

The method of constructing quasiconformal extension criteria is based on the
following result due to Becker (see \cite{Be1}, \cite{Be2} and also \cite%
{Be3}).

\begin{theo}
Suppose that $L(z,t)$ is a subordination chain. Consider 
\begin{equation*}
w(z,t)=\frac{p(z,t)-1}{p(z,t)+1}\;,\;z\in \mathcal{U}\;,\;t\geq 0
\end{equation*}%
where $p(z,t)$ is defined by (\ref{2.1}). If 
\begin{equation*}
|w(z,t)|\leq k,\;0\leq k<1
\end{equation*}%
for all $z\in \mathcal{U}$ and $t\geq 0$, then $L(z,t)$ admits a continuous
extension to $\bar{\mathcal{U}}$ for each $t\geq 0$ and the function $F(z,%
\bar{z})$ defined by

\begin{equation*}
F(z,\bar{z})=\left\{ 
\begin{array}{ll}
L(z,0) & \mbox{,\;if \; $|z|<1$} \\ 
L\left( \frac{z}{|z|},\log |z|\right) & \mbox{,\;if \;
$|z|\geq1$.}%
\end{array}%
\right.
\end{equation*}%
is a k-quasiconformal extension of $L(z,0)$ to $\mathbb{C}$.
\end{theo}

Examples of quasiconformal extension criteria can be found in \cite{Ah}, 
\cite{AnHi}, \cite{Bet}, \cite{Kr}, \cite{Pf} and more recently in \cite{Ho1}%
, \cite{Ho2}, \cite{Ho3}.

\section{Univalence criteria}

In this section, making use of Theorem 2.1, we obtain certain sufficient
conditions for the univalence of an integral operator.

\begin{theo}
Let $f,g,\phi \in \mathcal{A},\;g(z)\neq 0,\;\phi (z)\neq 0$ in $\mathcal{U}$%
. Let also $m\in \mathbb{R}_{+}$ and $\alpha ,\beta ,\gamma \in \mathbb{C}$.
If 
\begin{equation}
\left\vert \frac{(1-|z|^{(m+1)\gamma })}{\gamma }\left[ \alpha \frac{%
zf^{\prime \prime }(z)}{f^{\prime }(z)}+\beta \left( \frac{zg^{\prime }(z)}{%
g(z)}-\frac{z\phi ^{\prime }(z)}{\phi (z)}\right) \right] -\frac{m-1}{2}%
\right\vert \leq \frac{m+1}{2}  \label{3.1}
\end{equation}%
holds $z\in \mathcal{U}$, then the function $\mathcal{F}_{\alpha ,\beta
,\gamma }$ defined by 
\begin{equation}
\mathcal{F}_{\alpha ,\beta ,\gamma }(z)=\left[ {\gamma
\int\limits_{0}^{z}u^{\gamma -1}}\left( f^{\prime }(u)\right) ^{\alpha
}\left( \frac{g(u)}{\phi (u)}\right) ^{\beta }du\right] ^{1/\gamma }\text{ }%
z\in \mathcal{U},  \label{3.2}
\end{equation}%
where the principal branch is intended, is analytic and univalent in $%
\mathcal{U}$.
\end{theo}

\begin{proof}
Let $a$ be a positive real number. We are going to prove that there exists $%
r\in (0,1]$ such that the function $L:\mathcal{U}_{r}\times \lbrack 0,\infty
)\rightarrow \mathbb{C}$, defined by 
\begin{equation}
L(z,t)=\left\{ \gamma \int\limits_{0}^{e^{-at}z}u^{\gamma -1}(f^{\prime
}(u))^{\alpha }\left( \frac{g(u)}{\phi (u)}\right) ^{\beta }du+(e^{mat\gamma
}-e^{-at\gamma })z^{\gamma }(f^{\prime }(e^{-at}z))^{\alpha }\left( \frac{%
g(e^{-at}z)}{\phi (e^{-at}z)}\right) ^{\beta }\right\} ^{1/\gamma }
\label{3.3}
\end{equation}%
is analytic in $\mathcal{U}_{r}$ for all $t\in \lbrack 0,\infty )$ and
satisfies the conditions of Theorem 2.1. Since $f,g,\phi \in \mathcal{A}$,
there exists a disk $\mathcal{U}_{r_{1}},\;0<r_{1}\leq 1$ in which the
function 
\begin{equation*}
h(z)=(f^{\prime }(z))^{\alpha }\left( \frac{g(z)}{\phi (z)}\right) ^{\beta }
\end{equation*}%
is analytic. The powers are considered with their principal branches. The
function $h(z)$ is analytic and does not vanish in $\mathcal{U}_{r_{1}}$.

Consider the function 
\begin{equation*}
h_{1}(z,t)=\gamma \int\limits_{0}^{e^{-at}z}u^{\gamma -1}h(u)du\;,\;z\in 
\mathcal{U}_{r_{1}}\;,\;t\geq 0.
\end{equation*}%
We can write 
\begin{equation*}
h_{1}(z,t)=z^{\gamma }h_{2}(z,t)
\end{equation*}%
where $h_{2}(z,t)$ is analytic in $\mathcal{U}_{r_{1}}$ for all $t\geq 0$.
It follows that the function 
\begin{equation*}
h_{3}(z,t)=h_{2}(z,t)+(e^{mat\gamma }-e^{-at\gamma })h(e^{-at}z)
\end{equation*}%
is also analytic in $\mathcal{U}_{r_{1}}$ and 
\begin{equation*}
h_{3}(0,t)=e^{mat\gamma }.
\end{equation*}%
Since $h_{3}(0,0)=1,h_{3}(z,t)\neq 0$ for $t\geq 0$ and $\lim_{t\rightarrow
\infty }|h_{3}(z,t)|=\infty $ there exists a disk $\mathcal{U}%
_{r_{2}},\;0<r_{2}\leq r_{1}.$ Therefore we can choose a uniform and
analytic branch of $[h_{3}(z,t)]^{1/\gamma }$ in $\mathcal{U}_{r_{2}}$ which
will be denoted by $h_{4}(z,t)$. Now, the function defined by (\ref{3.3})
can be rewritten as 
\begin{equation}
L(z,t)=zh_{4}(z,t)=a_{1}(t)z+...,\;z\in \mathcal{U}_{r_{2}}\;\text{and}%
\;t\geq 0  \label{3.4}
\end{equation}%
where $a_{1}(t)=e^{mat}.$ Moreover $L(z,t)$ is analytic in $\mathcal{U}%
_{r_{2}}$ for all $t\geq 0$.

Let $r_{3}\in (0,r_{2}]$ and let $K=\left\{ z\in \mathbb{C}:|z|\leq
r_{3}\right\} $. Since the function $L(z,t)$ is analytic in $\mathcal{U}%
_{r_{2}}$, there exists $M>0$ such that 
\begin{equation*}
\left\vert \frac{L(z,t)}{a_{1}(t)}\right\vert \leq M\;\text{for}\;z\in K\;%
\text{and}\;t\geq 0.
\end{equation*}%
Thus, $\left\{ \frac{L(z,t)}{a_{1}(t)}\right\} _{t\geq 0}$ forms a normal
family in $\mathcal{U}_{r_{2}}$.

From (\ref{3.4}) we obtain that $\left\{ \frac{\partial L(z,t)}{\partial t}%
\right\} $ is analytic in $\mathcal{U}_{r_{2}}$. It follows that $\left\vert 
\frac{\partial L(z,t)}{\partial t}\right\vert $ is bounded on $[0,T]$ for
any fixed $T>0$ and $z\in \mathcal{U}_{r_{2}}$. Therefore, the function $%
L(z,t)$ is locally absolutely continuous on $[0,\infty )$, locally uniform
with respect to $\mathcal{U}_{r_{2}}$.

For $0<r\leq r_{2}$ and $t\geq 0$, consider the function $p:\mathcal{U}%
_{r}\times \lbrack 0,\infty )\rightarrow \mathbb{C}$ defined by 
\begin{equation*}
p(z,t)={z\frac{\partial L(z,t)}{\partial z}}\diagup \frac{\partial L(z,t)}{%
\partial t}.
\end{equation*}%
In order to prove that the function $p(z,t)$ is analytic and has positive
real part in $\mathcal{U}$, we will show that the function 
\begin{equation*}
w(z,t)=\frac{p(z,t)-1}{p(z,t)+1}
\end{equation*}%
is analytic in $\mathcal{U}$ and 
\begin{equation}
|w(z,t)|<1\;,\;\text{for all}\;z\in \mathcal{U}\;\text{and}\;t\geq 0.
\label{3.5}
\end{equation}%
Lengthy but elementary calculation gives 
\begin{equation}
w(z,t)=\frac{(1+a)\mathcal{G}(z,t)+1-ma}{(1-a)\mathcal{G}(z,t)+1+ma},
\label{3.6}
\end{equation}%
where 
\begin{equation}
\mathcal{G}(z,t)=\frac{1}{\gamma }\left[ \alpha \frac{e^{-at}zf^{\prime
\prime }(e^{-at}z)}{f^{\prime }(e^{-at}z)}+\beta \left( \frac{%
e^{-at}zg^{\prime }(e^{-at}z)}{g(e^{-at}z)}-\frac{e^{-at}z\phi ^{\prime
}(e^{-at}z)}{\phi (e^{-at}z)}\right) \right] \left( 1-e^{-(m+1)at\gamma
}\right)  \label{3.7}
\end{equation}%
for $z\in \mathcal{U}$ and $t\geq 0$.

Inequality (\ref{3.5}) is therefore, equivalent to 
\begin{equation}
\left\vert \mathcal{G}(z,t)-\frac{m-1}{2}\right\vert <\frac{m+1}{2}\;,\;z\in 
\mathcal{U}\;,\;t\geq 0.  \label{3.8}
\end{equation}%
For $t=0$ the last inequality holds. Define 
\begin{equation}
\mathcal{H}(z,t)=\mathcal{G}(z,t)-\frac{m-1}{2}\;,\;z\in \mathcal{U}%
\;,\;t\geq 0.  \label{3.9}
\end{equation}%
Since $|e^{-at}z|\leq |e^{-at}|=e^{-at}<1$ for all $z\in \bar{\mathcal{U}}%
=\left\{ z\in \mathbb{C}:|z|\leq 1\right\} $ and $t>0$, we have that $%
\mathcal{H}(z,t)$ is analytic in $\bar{\mathcal{U}}$ for every $t>0$. Making
use of the maximum modulus principle, we obtain that, for each arbitrary
fixed $t>0$, there exists $\theta (t)\in \mathbb{R}$ such that 
\begin{equation*}
|\mathcal{H}(z,t)|<\max_{|z|=1}|\mathcal{H}(z,t)|=|\mathcal{H}(e^{i\theta
},t)|\;\;\text{for all}\;z\in \mathcal{U}.
\end{equation*}%
Let $u=e^{-at}e^{i\theta }$. Then $|u|=e^{-at}$ and $%
e^{-(m+1)at}=(e^{-at})^{(m+1)}=|u|^{m+1}$. Therefore 
\begin{equation*}
|\mathcal{H}(e^{i\theta },t)|=\left\vert \frac{(1-|u|^{(m+1)\gamma })}{%
\gamma }\left[ \alpha \frac{uf^{\prime \prime }(u)}{f^{\prime }(u)}+\beta
\left( \frac{ug^{\prime }(u)}{g(u)}-\frac{u\phi ^{\prime }(u)}{\phi (u)}%
\right) \right] -\frac{m-1}{2}\right\vert .
\end{equation*}%
Inequality (\ref{3.1}) from hypothesis implies 
\begin{equation}
|\mathcal{H}(e^{i\theta },t)|\leq \frac{m+1}{2}.  \label{3.10}
\end{equation}%
From (\ref{3.10}) it follows that inequality (\ref{3.8}) is satisfied for
all $z\in \mathcal{U}$ and $t\geq 0$.

Since all the conditions of Theorem 2.1 are satisfied, we obtain that the
function $L(z,t)$ has an analytic and univalent extension to the whole unit
disk $\mathcal{U}$, for all $t\geq 0$. If $t=0$ we have $L(z,0)=\mathcal{F}%
_{\alpha ,\beta ,\gamma }(z)$ and therefore, our integral operator $\mathcal{%
F}_{\alpha ,\beta ,\gamma }$ is analytic and univalent in $\mathcal{U}$.
\end{proof}

Making use of Theorem 3.1, we derive another univalence criterion for the
integral operator $\mathcal{F}_{\alpha ,\beta ,\gamma }$.

\begin{theo}
Let $f,g,\phi \in \mathcal{A},\;g(z)\neq 0,\;\phi (z)\neq 0$. Let also $%
\alpha ,\beta ,\gamma \in \mathbb{C}$ with $\Re \gamma >0$ and $m\in \mathbb{%
R}_{+},\;m\geq 1$. If 
\begin{equation*}
\frac{1-|z|^{(m+1)\Re \gamma }}{\Re \gamma }\left\vert \alpha \frac{%
zf^{\prime \prime }(z)}{f^{\prime }(z)}+\beta \left( \frac{zg^{\prime }(z)}{%
g(z)}-\frac{z\phi ^{\prime }(z)}{\phi (z)}\right) \right\vert \leq 1
\end{equation*}%
holds for $z\in \mathcal{U}$ then, the function $\mathcal{F}_{\alpha ,\beta
,\gamma }$ defined by (\ref{3.2}) is analytic and univalent in $\mathcal{U}$.
\end{theo}

\begin{proof}
It can be proved (see \cite{Pa2}) that for $z\in \mathcal{U}\setminus
\left\{ 0\right\} ,\;\Re \gamma >0$ and $m\in \mathbb{R}_{+}$ 
\begin{equation*}
\left\vert \frac{1-|z|^{(m+1)\gamma }}{\gamma }\right\vert \leq \frac{%
1-|z|^{(m+1)\Re \gamma }}{\Re \gamma }.
\end{equation*}%
For $m\geq 1$, we have 
\begin{equation*}
\left\vert \frac{1-|z|^{(m+1)\gamma }}{\gamma }\left[ \alpha \frac{%
zf^{\prime \prime }(z)}{f^{\prime }(z)}+\beta \left( \frac{zg^{\prime }(z)}{%
g(z)}-\frac{z\phi ^{\prime }(z)}{\phi (z)}\right) \right] -\frac{m-1}{2}%
\right\vert
\end{equation*}%
\begin{equation*}
\leq \left\vert \frac{1-|z|^{(m+1)\gamma }}{\gamma }\left[ \alpha \frac{%
zf^{\prime \prime }(z)}{f^{\prime }(z)}+\beta \left( \frac{zg^{\prime }(z)}{%
g(z)}-\frac{z\phi ^{\prime }(z)}{\phi (z)}\right) \right] \right\vert +\frac{%
m-1}{2}
\end{equation*}%
\begin{equation*}
\leq \frac{1-|z|^{(m+1)\Re \gamma }}{\Re \gamma }\left\vert \alpha \frac{%
zf^{\prime \prime }(z)}{f^{\prime }(z)}+\beta \left( \frac{zg^{\prime }(z)}{%
g(z)}-\frac{z\phi ^{\prime }(z)}{\phi (z)}\right) \right\vert +\frac{m-1}{2}
\end{equation*}%
\begin{equation*}
\leq 1+\frac{m-1}{2}=\frac{m+1}{2}.
\end{equation*}%
Since inequality (\ref{3.1}) is satisfied, making use of Theorem 3.1, we can
conclude that the function $\mathcal{F}_{\alpha ,\beta ,\gamma }$ is
analytic and univalent in $\mathcal{U}$.
\end{proof}

\begin{example}
Let $\alpha ,\beta ,\gamma $ be three complex numbers such that $\Re \gamma
>0$ and $\Re \gamma \geq |\alpha |+|\beta |$. Then, the function 
\begin{equation*}
\mathcal{F}_{\alpha ,\beta ,\gamma }(z)=z\left[ _{2}F_{1}(\gamma ,-(\alpha
+\beta );1+\gamma ;-\frac{z}{2})\right] ^{1/\gamma }
\end{equation*}%
is univalent in $\mathcal{U}$. The symbol $_{2}F_{1}(a,b;c;z)$ denotes the
well known hypergeometric function.
\end{example}

\begin{proof}
Set $f(z)=z+\displaystyle\frac{z^{2}}{4},\;g(z)=z+\frac{z^{2}}{2},\;z\in 
\mathcal{U}$ and $\phi (z)=z,\;z\in \mathcal{U}$ in Theorem 3.2. Making use
of triangle inequality, we have 
\begin{equation*}
\frac{1-|z|^{(m+1)\Re \gamma }}{\Re \gamma }\left\vert \alpha \frac{%
zf^{\prime \prime }(z)}{f^{\prime }(z)}+\beta \left( \frac{zg^{\prime }(z)}{%
g(z)}-\frac{z\phi ^{\prime }(z)}{\phi (z)}\right) \right\vert
\end{equation*}%
\begin{equation*}
=\frac{1-|z|^{(m+1)\Re \gamma }}{\Re \gamma }\left\vert \alpha \frac{z}{z+2}%
+\beta \left( \frac{2z+2}{z+2}-1\right) \right\vert
\end{equation*}%
\begin{equation*}
\leq \frac{1-|z|^{(m+1)\Re \gamma }}{\Re \gamma }\frac{|z|}{2-|z|}(|\alpha
|+|\beta |)<\frac{1}{\Re \gamma }(|\alpha |+|\beta |)\leq 1.
\end{equation*}%
The last inequality follows from $1-|z|^{(m+1)\Re \gamma }<1,\;\displaystyle%
\frac{|z|}{2-|z|}<1,\;z\in \mathcal{U}$ and $\Re \gamma \geq |\alpha
|+|\beta |$. Since all the conditions of Theorem 3.2 are satisfied, we
obtain that the function 
\begin{equation*}
\mathcal{F}_{\alpha ,\beta ,\gamma }(z)=\left[ \gamma
\int\limits_{0}^{z}u^{\gamma -1}\left( 1+\frac{u}{2}\right) ^{\alpha }\left(
1+\frac{u}{2}\right) ^{\beta }du\right] ^{1/\gamma }
\end{equation*}%
is univalent in $\mathcal{U}$. With the substitution $u=tz$ the function $%
\mathcal{F}_{\alpha ,\beta ,\gamma }(z)$ becomes 
\begin{equation*}
\mathcal{F}_{\alpha ,\beta ,\gamma }(z)=z\left[ \gamma
\int\limits_{0}^{1}t^{\gamma -1}\left( 1+t\frac{z}{2}\right) ^{\alpha +\beta
}dt\right] ^{1/\gamma }=z\left[ _{2}F_{1}(\gamma ,-(\alpha +\beta );1+\gamma
;-\frac{z}{2})\right] ^{1/\gamma }.
\end{equation*}%
With this, the proof is complete.

Certain particular cases of Theorem 3.1 and Theorem 3.2 respectively, are
listed below.
\end{proof}

If in Theorem 3.1 we consider $\alpha =\beta ,\;g(z)=z$ and $\phi =f$, we
obtain the following univalence condition.

\begin{cor}
Let $f\in \mathcal{A}$ and $m\in \mathbb{R}_{+}$. If 
\begin{equation*}
\left\vert \alpha \frac{1-|z|^{(m+1)\gamma }}{\gamma }\left[ 1+\frac{%
zf^{\prime \prime }(z)}{f^{\prime }(z)}-\frac{zf^{\prime }(z)}{f(z)}\right] -%
\frac{m-1}{2}\right\vert \leq \frac{m+1}{2}
\end{equation*}%
holds $z\in \mathcal{U}$ then the function $\mathcal{F}_{\alpha ,\gamma }(z)$
defined by 
\begin{equation}
\mathcal{F}_{\alpha ,\gamma }(z)=\left[ \gamma \int\limits_{0}^{z}u^{\gamma
-1}\left( \frac{uf^{\prime }(u)}{f(u)}\right) ^{\alpha }du\right] ^{1/\gamma
}  \label{3.11}
\end{equation}%
is analytic and univalent in $\mathcal{U}$.
\end{cor}

If we take $\alpha=\gamma=1$ then, the integral operator $\mathcal{F}%
_{\alpha,\gamma}(z)$ defined by (\ref{3.11}) reduces to the integral
operator considered by Danikas and Ruscheweyh in \cite{DaRu}.

An improvement of Becker's univalence criterion (see\cite{Be1}) which was
obtained by Pascu can be derived from Theorem 3.2 for $\alpha=1,\;g=\phi$
and $m=1$.

\begin{cor}
\textrm{(\cite{Pa2})} Let $f\in \mathcal{A}$ and $\gamma \in \mathbb{C}%
,\;\Re \gamma >0$. If 
\begin{equation*}
\frac{1-|z|^{2\Re \gamma }}{\Re \gamma }\left\vert \frac{zf^{\prime \prime
}(z)}{f^{\prime }(z)}\right\vert \leq 1\;,\;z\in \mathcal{U}
\end{equation*}%
then the integral operator%
\begin{equation*}
\mathcal{F}_{\gamma }(z)=\left[ \gamma \int\limits_{0}^{z}u^{\gamma
-1}f^{\prime }(u)du\right] ^{1/\gamma }
\end{equation*}%
is analytic and univalent in $\mathcal{U}$.
\end{cor}

\section{Quasiconformal extension criterion}

In this section we will refine the univalence condition given in Theorem 3.1
to a quasiconformal extension criterion.

\begin{theo}
Let $f,g,\phi \in \mathcal{A},\;g(z)\neq 0,\;\phi (z)\neq 0$. Let also $m\in 
\mathbb{R}_{+}$, $\alpha ,\beta ,\gamma \in \mathbb{C}$ and $k\in \lbrack
0,1)$. If 
\begin{equation}
\left\vert \frac{(1-|z|^{(m+1)\gamma })}{\gamma }\left[ \alpha \frac{%
zf^{\prime \prime }(z)}{f^{\prime }(z)}+\beta \left( \frac{zg^{\prime }(z)}{%
g(z)}-\frac{z\phi ^{\prime }(z)}{\phi (z)}\right) \right] -\frac{m-1}{2}%
\right\vert \leq k\frac{m+1}{2}  \label{4.1}
\end{equation}%
is true for $z\in \mathcal{U}$ then, the function $\mathcal{F}_{\alpha
,\beta ,\gamma }$ given by (\ref{3.2}) has a quasiconformal extension to $%
\mathbb{C}$.
\end{theo}

\begin{proof}
In the proof of Theorem 3.1 has been proved that the function $L(z,t)$ given
by (\ref{3.3}) is a subordination chain in $\mathcal{U}$. Applying Theorem
2.2 to the function $w(z,t)$ given by (\ref{3.6}), we obtain that the
assumption 
\begin{equation}
\left\vert \frac{(1+a)\mathcal{G}(z,t)+1-ma}{(1-a)\mathcal{G}(z,t)+1+ma}%
\right\vert <l,\;z\in \mathcal{U},\;t\geq 0\;\text{and}\;l\in \lbrack 0,1)
\label{4.2}
\end{equation}%
where $\mathcal{G}(z,t)$ is defined by (\ref{3.7}), implies $l$%
-quasiconformal extensibility of $\mathcal{F}_{\alpha ,\beta ,\gamma }$.

Lenghty but elementary calculation shows that the last inequality (\ref{4.2}%
) is equivalent to 
\begin{equation}
\left\vert \mathcal{G}(z,t)-\frac{a(1+l^{2})(m-1)+(1-l^{2})(ma^{2}-1)}{%
2a(1+l^{2})+(1-l^{2})(1+a^{2})}\right\vert \leq \frac{2al(1+m)}{%
2a(1+l^{2})+(1-l^{2})(1+a^{2})}.  \label{4.3}
\end{equation}

It is easy to check that, under the assumption (\ref{4.1}) we have 
\begin{equation}
\left\vert \mathcal{G}(z,t)-\frac{m-1}{2}\right\vert \leq k\frac{m+1}{2}.
\label{4.4}
\end{equation}

Consider the two disks $\Delta $ and $\Delta ^{\prime }$ defined by (\ref%
{4.3}) and (\ref{4.4}) respectively, where $\mathcal{G}(z,t)$ is replaced by
a complex variable $\zeta $. Our theorem will be proved if we find the
smallest $l\in \lbrack 0,1)$ for which $\Delta ^{\prime }$ is contained in $%
\Delta $. This will be so if and only if the distance apart of the centers
plus the smallest radius is equal, at most, to the largest radius. So, we
are required to prove that 
\begin{equation*}
\left\vert \frac{a(1+l^{2})(m-1)+(1-l^{2})(ma^{2}-1)}{%
2a(1+l^{2})+(1-l^{2})(1+a^{2})}-\frac{m-1}{2}\right\vert +k\frac{m+1}{2}\leq 
\frac{2al(1+m)}{2a(1+l^{2})+(1-l^{2})(1+a^{2})}
\end{equation*}%
or equivalently 
\begin{equation}
\frac{(1-l^{2})|1-a^{2}|}{2[2a(1+l^{2})+(1-l^{2})(1+a^{2})]}\leq \frac{2al}{%
2a(1+l^{2})+(1-l^{2})(1+a^{2})}-\frac{k}{2}  \label{4.5}
\end{equation}%
with the condition 
\begin{equation}
\frac{2al}{2a(1+l^{2})+(1-l^{2})(1+a^{2})}-\frac{k}{2}\geq 0.  \label{4.6}
\end{equation}

We will solve inequalities (\ref{4.5}) and (\ref{4.6}) for $1-a^{2}>0$. In a
similar way they can be solved for $1-a^{2}<0$.

The solutions of the quadratic equation obtained from (\ref{4.4}), where
instead of inequality sign we put equal, are: 
\begin{equation*}
L_{1}=\frac{(1-a)^{2}+k(1-a^{2})}{1-a^{2}+k(1-a)^{2}}\;,\;L_{2}=-\frac{%
(1+a)^{2}+k(1-a^{2})}{1-a^{2}+k(1-a)^{2}}.
\end{equation*}

Therefore, the solution of inequality (\ref{4.5}) is $l\leq L_{2}$ and $%
L_{1}\leq l$. Since $L_{2}<0$ it remains $L_{1}\leq l$.

After similar calculations, from inequality (\ref{4.6}), we get $l\leq%
\mathcal{L}_{2}$ and $\mathcal{L}_{1}\leq l$, where 
\begin{equation*}
\mathcal{L}_{1}=\frac{-2a+\sqrt{4a^{2}+(1-a^{2})^{2}k^{2}}}{k(1-a)^{2}}\;,\;%
\mathcal{L}_{2}=\frac{-2a-\sqrt{4a^{2}+(1-a^{2})^{2}k^{2}}}{k(1-a)^{2}}.
\end{equation*}

Since $\mathcal{L}_{2}<0$ it follows $\mathcal{L}_{1}\leq l$.

It can be checked, eventually by using Mathematica program, that $\mathcal{L}%
_{1}\leq L_{1}$ and thus $L_{1}\leq l<1.$ If $a=1$, both inequalities (4.5)
and (4.6) reduce to $k\leq l$.

Consequently, we proved that the assumption (4.1) implies the existence of
an $l$-quasiconformal extension of $\mathcal{F}_{\alpha ,\beta ,\gamma }$ to 
$\mathbb{C}$, which is given by 
\begin{equation*}
l=\left\{ 
\begin{array}{ll}
\frac{(1-a)^{2}+k|1-a^{2}|}{|1-a^{2}|+k(1-a)^{2}}, & a\in (0,\infty
)\setminus \left\{ 1\right\} \\ 
k, & a=1.%
\end{array}%
\right.
\end{equation*}

Therefore $L_{1}\leq l<1$ and the proof is complete.
\end{proof}

\end{document}